\documentclass[reqno, 11pt]{amsart}
\pdfoutput=1
\usepackage[x11names]{xcolor}
\usepackage{amsmath, amssymb}
\usepackage{amsthm}
\usepackage{mathtools}
\usepackage{hyperref}
\usepackage[shortlabels]{enumitem}
\usepackage[portrait, margin=1in]{geometry}

\usepackage{subcaption}

\hypersetup{
    colorlinks = true,
    linktoc = all,
    urlcolor = Blue3!50!Blue4,
    linkcolor = Blue3!50!Blue4,
    citecolor = Blue3!50!Blue4
}

\usepackage{tikz}
\usetikzlibrary{trees}
\usetikzlibrary{arrows.meta}
\usetikzlibrary{decorations.pathmorphing}
\tikzset{snaked/.style = {decorate, decoration=snake}}
\usetikzlibrary{positioning}
\usetikzlibrary{decorations.pathreplacing}
\tikzset{sdot/.style = {fill, circle, inner sep = 1.5pt}}
\usetikzlibrary{calc}

\colorlet{myblue}{Blue4!70!DarkSlateGray3}
\colorlet{mypurple}{Purple4!70!Blue4}

\newtheorem{theorem}{Theorem}[section]

\newtheorem{lemma}[theorem]{Lemma}

\newtheorem{claim}[theorem]{Claim}
\newtheorem*{theorem*}{Theorem}
\newtheorem*{corollary*}{Corollary}
\newtheorem*{lemma*}{Lemma}
\newtheorem*{prop*}{Proposition}

\newtheorem*{fact*}{Fact}
\newtheorem*{claim*}{Claim}

\theoremstyle{definition}

\newtheorem*{example*}{Example}
\newtheorem*{defn*}{Definition}
\newtheorem*{remark*}{Remark}

\numberwithin{equation}{section}

\newcommand{\sceil}[1]{\lceil #1 \rceil}
\newcommand{\sdbold}[1]{\textbf{\textsf{#1}}}

\newcommand{\abs}[1]{\left\lvert #1 \right\rvert}
\newcommand{\sabs}[1]{\lvert #1 \rvert}

\newcommand{\ora}[1]{\overrightarrow{#1}}

\newcommand{\NN}{\mathbb N}

\newcommand{\RR}{\mathbb R}

\newcommand{\cC}{\mathcal C}

\newcommand{\cI}{\mathcal I}

\newcommand{\cP}{\mathcal P}

\newcommand{\cR}{\mathcal R}

\newcommand{\cV}{\mathcal V}

\newcommand{\cZ}{\mathcal Z}

\newcommand{\eps}{\varepsilon}

\newcommand{\case}[2]{\textcolor{Purple4!70!Blue4}{\sdbold{Case #1} (#2).}}

\newcommand{\deftext}[1]{\emph{#1}}
\newcommand{\Strip}{\mathsf{Strip}}
\newcommand{\sfE}{\mathsf{E}}

\title{Distinct distances between a line and strip}
\author{Sanjana Das}
\address{Department of Mathematics, Massachusetts Institute of Technology, MA, USA}
\email{\href{mailto:sanjanad@mit.edu}{sanjanad@mit.edu}}
\author{Adam Sheffer}
\address{Department of Mathematics, Baruch College, City University of New York, NY, USA}
\email{\href{mailto:adamsh@gmail.com}{adamsh@gmail.com}}
\date{April 6, 2025}

\begin{document}

\begin{abstract}
    We introduce a new type of distinct distances result: a lower bound on the number of distances between points on a line and points on a two-dimensional strip. This can be seen as a generalization of the well-studied problems of distances between points on two lines or curves. Unlike these existing problems, this new variant only makes sense if the points satisfy an additional spacing condition. 

    Our work can also be seen as an exploration of the proximity technique that was recently introduced by Solymosi and Zahl. This technique lies at the heart of our analysis. 
\end{abstract}

\maketitle

\section{Introduction} \label{sec:intro}

\subsection{Background} \label{subsec:background}

The study of distinct distances began with the following question of Erd\H{o}s \cite{Erd46}: What is the minimum possible number of distinct distances that $n$ points in $\RR^2$ can span? For example, $n$ equally spaced points on a line span $n - 1$ distinct distances. In \cite{Erd46}, Erd\H{o}s showed that it is possible to do a little better --- a $\sqrt{n} \times \sqrt{n}$ lattice spans $\Theta(n/\sqrt{\log n})$ distances. He also proved that any $n$ points span $\Omega(n^{1/2})$ distances. Since then, this lower bound has seen a series of improvements; finally, Guth and Katz \cite{GK15} proved that any $n$ points span $\Omega(n/\log n)$ distances, which resolves the problem up to a factor of $\sqrt{\log n}$. 

While the above problem is nearly resolved, it is just one out of many distinct distances problems that Erd\H{o}s introduced, and most other variants are far from understood. One interesting variant of this problem, posed by Purdy (see \cite[Section 5.5]{BPS05}), is as follows: For two lines $\ell_1$ and $\ell_2$ and two sets of $n$ points $\cP_1 \subseteq \ell_1$ and $\cP_2 \subseteq \ell_2$, what is the minimum number of distinct distances between points in $\cP_1$ and points in $\cP_2$? 

For two points $a, p \in \RR^2$, we write $\abs{ap}$ to denote the distance between $a$ and $p$. For two sets of points $\cP_1, \cP_2 \subseteq \RR^2$, we write \[\Delta(\cP_1, \cP_2) = \{\abs{ap} \mid a \in \cP_1, \, p \in \cP_2\}\] to denote the set of distances between a point in $\cP_1$ and a point in $\cP_2$. Then Purdy's question asks for the minimum value of $\abs{\Delta(\cP_1, \cP_2)}$ over all $n$-element sets $\cP_1 \subseteq \ell_1$ and $\cP_2 \subseteq \ell_2$. 

If $\ell_1$ and $\ell_2$ are parallel or orthogonal, the answer to this question is $\Theta(n)$: If $\ell_1$ and $\ell_2$ are parallel, then we can take $n$ equally spaced points on each line. If $\ell_1$ and $\ell_2$ are orthogonal, then taking them to be the $x$-axis and $y$-axis of a coordinate system, we can take $\cP_1 = \{(\sqrt{i}, 0) \mid i \in [n]\}$ and $\cP_2 = \{(0, \sqrt{i}) \mid i \in [n]\}$.

Purdy \cite[Section 5.5]{BPS05} conjectured that these special cases are the \emph{only} ones where the answer is $\Theta(n)$ --- more precisely, he conjectured that if $\ell_1$ and $\ell_2$ are not parallel or orthogonal, then we must have $\abs{\Delta(\cP_1, \cP_2)} = \omega(n)$. This was proven by Elekes and R\'onyai \cite{ER00}. Since then, there has been a series of quantitative improvements on the lower bound: 
\begin{itemize}
    \item Elekes \cite{Ele99} proved that $\abs{\Delta(\cP_1, \cP_2)} = \Omega(n^{5/4})$.
    \item Sharir, Sheffer, and Solymosi \cite{SSS13} proved that $\abs{\Delta(\cP_1, \cP_2)} = \Omega(n^{4/3})$. More generally, they showed that if $\abs{\cP_1} = m$ and $\abs{\cP_2} = n$, then 
    \begin{equation}
        \abs{\Delta(\cP_1, \cP_2)} = \Omega(\min\{m^{2/3}n^{2/3}, m^2, n^2\}). \label{eqn:sss13}
    \end{equation}
    
    \item Recently, Solymosi and Zahl \cite{SZ24} improved this to 
    \begin{equation}
        \abs{\Delta(\cP_1, \cP_2)} = \Omega(\min\{m^{3/4}n^{3/4}, m^2, n^2\}). \label{eqn:sz24}
    \end{equation}
    (In the case $m = n$, this gives $\abs{\Delta(\cP_1, \cP_2)} = \Omega(n^{3/2})$.) 
\end{itemize} 

The main idea behind the proof of \eqref{eqn:sss13} is to consider the \deftext{distance energy} 
\begin{equation}
    \sfE(\cP_1, \cP_2) = \abs{\{(a, b, p, q) \in \cP_1^2 \times \cP_2^2 \mid \abs{ap} = \abs{bq}\}}. \label{eqn:ordinary-energy}
\end{equation} 
If $\abs{\Delta(\cP_1, \cP_2)}$ is small, then $\sfE(\cP_1, \cP_2)$ must be large; so it suffices to prove an \emph{upper} bound on $\sfE(\cP_1, \cP_2)$. Sharir, Sheffer, and Solymosi did so using incidence geometry: They set up a collection of points and curves in $\RR^2$, with one point for each $(a, b) \in \cP_1^2$ and one curve for each $(p, q) \in \cP_2^2$, such that $\sfE(\cP_1, \cP_2)$ is equal to the number of incidences between these points and curves. They then used incidence bounds to obtain an upper bound on $\sfE(\cP_1, \cP_2)$, and therefore a lower bound on $\abs{\Delta(\cP_1, \cP_2)}$. 

To obtain the improvement in \eqref{eqn:sz24}, Solymosi and Zahl introduced the technique of \emph{proximity}. Instead of working directly with the energy $\sfE(\cP_1, \cP_2)$ as defined in \eqref{eqn:ordinary-energy}, they worked with a `proximity-restricted' variant where we also require that $a$ be close to $b$ and that $p$ be close to $q$. The intuition behind why this leads to a better bound is that when setting up an incidence problem, we use pairs $(a, b) \in \cP_1^2$ to define points and pairs $(p, q) \in \cP_2^2$ to define curves; so these proximity conditions shrink the number of points and curves, and therefore the upper bound on the energy that we get from incidence bounds. 
    
These proximity conditions also shrink the energy itself. But loosely speaking, the idea is that among quadruples $(a, b, p, q)$ with $\abs{ap} = \abs{bq}$, the conditions that $a$ be close to $b$ and that $p$ be close to $q$ are very `well-correlated,' so that imposing \emph{both} proximity conditions does not shrink our lower bound on the energy in terms of $\abs{\Delta(\cP_1, \cP_2)}$ by much more than imposing just \emph{one} condition would. This means the proximity conditions shrink the upper bound on the energy by more than they shrink the lower bound in terms of $\abs{\Delta(\cP_1, \cP_2)}$, resulting in a stronger lower bound for $\abs{\Delta(\cP_1, \cP_2)}$. 

The above bounds show that the problem of distinct distances between two lines behaves quite differently in the cases where $\ell_1$ and $\ell_2$ are parallel or orthogonal and where they are not. However, the current best lower bound for the latter is still quite far from the best construction we know of, which has $\abs{\Delta(\cP_1, \cP_2)} = \Theta(n^2/\sqrt{\log n})$. This construction takes $\ell_1$ to be the $x$-axis and $\ell_2$ to be the line $y = x$, and sets $\cP_1 = \{(i, 0) \mid i \in [n]\}$ and $\cP_2 = \{(i, i) \mid i \in [n]\}$. This is illustrated in Figure \ref{fig:two-lines-construction}.

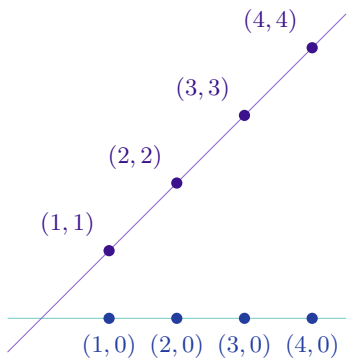
\begin{figure}[ht]
    \centering
    \begin{tikzpicture}[scale = 0.9]
        \draw [DarkSlateGray3] (-0.5, 0) -- (4.5, 0);
        \draw [MediumPurple3] (-0.5, -0.5) -- (4.5, 4.5);
        \foreach \i in {1, ..., 4} {
            \node [sdot, myblue, label = {below, font = \footnotesize, myblue}: {$(\i, 0)$}] at (\i, 0) {};
            \node [sdot, mypurple, label = {above left, font = \footnotesize, mypurple}: {$(\i, \i)$}] at (\i, \i) {};
        }
    \end{tikzpicture}
    \caption{A construction with $\Theta(n^2/\sqrt{\log n})$ distances.} \label{fig:two-lines-construction}
\end{figure}

Elekes \cite{Ele99} conjectured that this upper bound is closer to the truth --- specifically, he conjectured that for every $\eps > 0$, we must have $\abs{\Delta(\cP_1, \cP_2)} = \Omega(n^{2 - \eps})$.

\subsection{Our problem and main result}

We show how the proximity approach can be used to study a generalization of the problem of distinct distances between two lines. We still require $\cP_1$ to lie on a line, but we only require $\cP_2$ to lie in a two-dimensional \emph{strip}. To formalize this, for a curve $\gamma \subseteq \RR^2$ and $w > 0$, we define the \deftext{width-$w$ strip around $\gamma$} as the set \[\Strip_w(\gamma) = \{p \in \RR^2 \mid \text{exists $p^* \in \gamma$ with $\sabs{pp^*} \leq w$}\}.\] We fix two lines $\ell_1$ and $\ell_2$ and a parameter $w > 0$ (which we think of as a constant); we wish to understand the minimum possible value of $\abs{\Delta(\cP_1, \cP_2)}$ for point sets $\cP_1 \subseteq \ell_1$ and $\cP_2 \subseteq \Strip_w(\ell_2)$ of given sizes $m$ and $n$. (See Figure \ref{fig:problem-illustration}.)

For the condition that $\cP_2$ lies in a strip to be meaningful, we need to impose a spacing condition on the points. Otherwise, this problem would be equivalent to one where $\cP_1$ is restricted to a line but $\cP_2$ is allowed to be arbitrary --- one could take any configuration of points where $\cP_1$ lies on a line, shrink it down, and place it in a small region around the intersection of $\ell_1$ and $\ell_2$. So for $u > 0$, we say a collection of points $\cP$ is \deftext{$u$-spaced} if the distance between any two distinct points in $\cP$ is at least $u$. We will require both $\cP_1$ and $\cP_2$ to be $u$-spaced for some constant $u > 0$. 

\begin{figure}[ht]
    \centering 
    \begin{minipage}{0.47\textwidth}
        \centering
        \begin{tikzpicture}[scale = 0.7]
            \begin{scope}[rotate = 30]
                \fill [MediumPurple3!15, rounded corners = 0.3cm] (-1.8, -0.3) rectangle (5.3, 0.3);
            \end{scope}
            \draw [DarkSlateGray3] (-1.5, 0) -- (5, 0);
            \draw [MediumPurple3] (210:1.5) -- (30:5);
            \foreach \i in {1, 2.5, 3, 3.35, 3.95} {
                \node [sdot, myblue] at (\i, 0) {};
            }
            \foreach \i\j\k in {1.2/90/0.2, 2/285/0.15, -0.75/270/0.2, 3.5/90/0.18, 4.25/280/0.1} {
                \node [sdot, mypurple] at ($(30:\i) + (\j:\k)$) {};
            }
        \end{tikzpicture}

        \caption{We still require $\cP_1$ to lie on $\ell_1$, but we only require $\cP_2$ to lie on a strip around $\ell_2$.} \label{fig:problem-illustration}
    \end{minipage}
    \begin{minipage}{0.47\textwidth}
        \centering
        \begin{tikzpicture}[scale = 0.8]
            \begin{scope}[rotate = 30]
                \fill [MediumPurple3!15, rounded corners = 0.3cm] (-2.3, -0.3) rectangle (3.3, 0.3);
            \end{scope}
            \draw [DarkSlateGray3] (-2, 0) -- (3, 0);
            \draw [MediumPurple3] (210:2) -- (30:3);
            \filldraw [gray, fill opacity = 0.05, draw opacity = 0.5] (-0.3, -0.2) rectangle (0.3, 0.2);
            \fill [MediumPurple3!15] (1, -1) rectangle (2.5, -2);
            \draw [DarkSlateGray3] (1, -1.5) -- (2.5, -1.5);
            \foreach \i in {0.25, 0.6, 1, 1.35} {
                \node [sdot, myblue] at (\i + 1, -1.5) {};
            }
            \foreach \i\j in {0.3/0.3, 0.5/-0.27, 1.1/-0.3, 1.2/0.27} {
                \node [sdot, mypurple] at (\i + 1, -1.5 + \j) {};
            }
            \filldraw [gray, fill opacity = 0.05, draw opacity = 0.5] (1, -1) rectangle (2.5, -2);
            \node [rotate = -45, font = {\Huge}, gray!50] at (0.65, -0.6) {$\rightsquigarrow$};
        \end{tikzpicture}
        \caption{An illustration of why some spacing condition is necessary for the problem to be meaningful.} \label{fig:spacing-necessary}
    \end{minipage}
\end{figure}

Our main result is the following lower bound for this problem. 

\begin{theorem} \label{thm:line-vs-strip}
    Fix $u, w > 0$. Let $\ell_1$ and $\ell_2$ be two lines which are not parallel or orthogonal, and let $\cP_1 \subseteq \ell_1$ and $\cP_2 \subseteq \Strip_w(\ell_2)$ be $u$-spaced sets of points with $\abs{\cP_1} = m$ and $\abs{\cP_2} = n$. Then for every $\eps > 0$, we have \[\abs{\Delta(\cP_1, \cP_2)} = \Omega(\min\{m^{14/15}n^{8/15 - \eps}, m^{3/4}n^{3/4}, m^2, n^2\}),\] where the implicit constant depends on $\eps$, $u$, $w$, and the angle between $\ell_1$ and $\ell_2$.  
\end{theorem}

For context, Bruner and Sharir \cite{BS18} considered a more general variant of Purdy's problem where we still require that $\cP_1$ lies on a line $\ell_1$, but allow $\cP_2$ to be arbitrary. They showed that as long as no two points in $\cP_2$ lie on a line parallel or perpendicular to $\ell_1$, we must have 
\begin{equation}
    \abs{\Delta(\cP_1, \cP_2)} = \Omega(\min\{m^{10/11}n^{4/11 - \eps}, m^{2/3}n^{2/3}, m^2, n^2\}). \label{eqn:bs18}
\end{equation}
Their proof followed a similar framework to the proof of \eqref{eqn:sss13} --- they also considered the energy $\sfE(\cP_1, \cP_2)$ as defined in \eqref{eqn:ordinary-energy}, set up a collection of points and curves such that $\sfE(\cP_1, \cP_2)$ counted incidences between them, and used incidence bounds to obtain an upper bound on $\sfE(\cP_1, \cP_2)$. Our proof of Theorem \ref{thm:line-vs-strip} primarily involves showing that under the additional constraints in our setting (where $\cP_2$ lies on a strip and both point sets are reasonably spaced out), it is possible to incorporate Solymosi and Zahl's proximity technique into this argument, which allows us to improve this bound. 

\subsection{A statement for nonlinear strips} \label{subsec:nonlinear-strips}

We prove a more general version of Theorem \ref{thm:line-vs-strip} where instead of confining $\cP_2$ to a strip around a \emph{line}, we confine it to a strip around a curve satisfying certain technical conditions. 

Since rotations do not affect distances, we may assume that $\ell_1$ is the $x$-axis. We replace $\ell_2$ with a curve of the form $\{(f(y), y) \mid y \in \RR\}$ for some $f : \RR \to \RR$. 
 
\begin{itemize}
    \item For $s > 0$, we say $f$ is \deftext{$s$-Lipschitz} if for all $y_1, y_2 \in \RR$, we have $\abs{f(y_1) - f(y_2)} \leq s\abs{y_1 - y_2}$.
    \item For $k \in \NN$, we say $f$ is \deftext{$k$-nice} if for every $\delta \geq 0$, it is possible to partition $[-\delta, \delta]$ into $k$ sets $S_1$, \ldots, $S_k$ such that the function 
    \begin{equation*}
        \varphi_\delta(y) = f(y) + \sqrt{\delta^2 - y^2}
    \end{equation*}
    is monotone on each, and similarly it is possible to partition $[-\delta, \delta]$ into $k$ sets $T_1$, \ldots, $T_k$ such that the function 
    \begin{equation*}
        \psi_y(\delta) = f(y) - \sqrt{\delta^2 - y^2}
    \end{equation*}
    is monotone on each. (In this paper, we always use the words `monotone,' `increasing,' and `decreasing' in the \emph{weak} sense.)
\end{itemize} 

\begin{theorem} \label{thm:gen-curves}
    Fix $u, w, s > 0$ and $k \in \NN$. Let $f : \RR \to \RR$ be $s$-Lipschitz and $k$-nice, and let \[\cP_1 \subseteq \{(x, 0) \mid x \in \RR\} \quad \text{and} \quad \cP_2 \subseteq \Strip_w(\{(f(y), y) \mid y \in \RR\})\] be $u$-spaced sets of points with $\abs{\cP_1} = m$ and $\abs{\cP_2} = n$ such that no two points in $\cP_2$ have the same $x$-coordinate. Then for every $\eps > 0$, we have \[\abs{\Delta(\cP_1, \cP_2)} = \Omega(\min\{m^{14/15}n^{8/15 - \eps}, m^{3/4}n^{3/4}, m^2, n^2\}),\] where the implicit constant depends on $\eps$, $u$, $w$, $s$, and $k$. 
\end{theorem}

The definition of $k$-niceness may look somewhat strange (it comes from the proof --- we need it in order to make proximity work). However, this condition is satisfied by many natural functions. In particular, functions of the form $f(y) = sy$ (corresponding to lines) are $2$-nice. One way to see this is to rearrange the equation $z = sy \pm \sqrt{\delta^2 - y^2}$ to $(z - sy)^2 + y^2 = \delta^2$, so the graphs of $\varphi_\delta$ and $\psi_\delta$ are the `top' and `bottom' arcs (respectively) of the ellipse \[\{(y, z) \in \RR^2 \mid (z - sy)^2 + y^2 = \delta^2\},\] each of which can be cut into two monotone pieces. (See Figure \ref{fig:lines-nice}.)

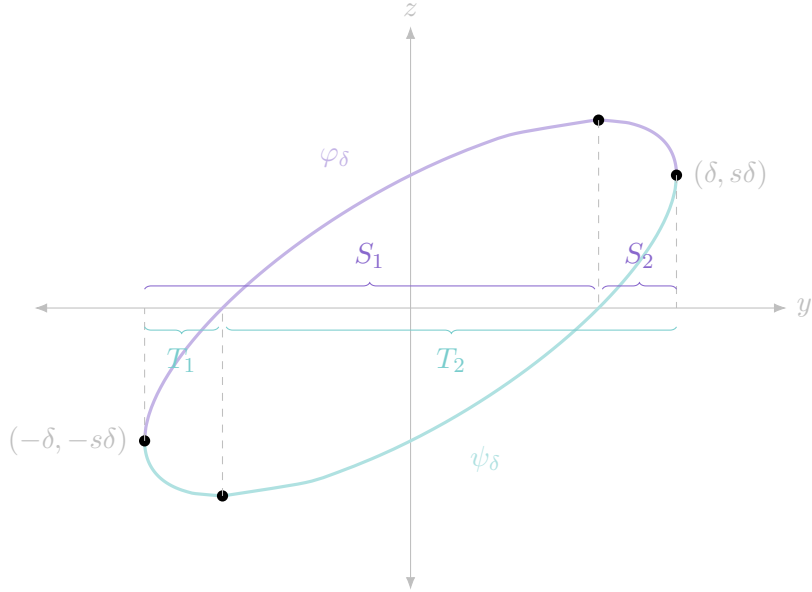
\begin{figure}[ht]
    \centering
    \begin{tikzpicture}[scale = 2.5]
        \draw [gray!50, Latex-Latex] (-2, 0) -- (2, 0) node [right] {$y$};
        \draw [gray!50, Latex-Latex] (0, -1.5) -- (0, 1.5) node [above] {$z$};
        \draw [domain = -1:{sqrt(2)/2}, smooth, DarkSlateGray3!60, very thick] plot ({\x + sqrt(1 - \x*\x)}, \x);
        \draw [domain = {sqrt(2)/2}:1, smooth, MediumPurple3!50, very thick] plot ({\x + sqrt(1 - \x*\x)}, \x);
        \draw [domain = -1:{-sqrt(2)/2}, smooth, DarkSlateGray3!60, very thick] plot ({\x - sqrt(1 - \x*\x)}, \x);
        \draw [domain = {-sqrt(2)/2}:1, smooth, MediumPurple3!50, very thick] plot ({\x - sqrt(1 - \x*\x)}, \x);
        \node [sdot] at (-1, -1) {};
        \node [sdot, label = {right, gray!50}: {$(\delta, s\delta)$}] at ({sqrt(2)}, {sqrt(2)/2}) {};
        \node [sdot] at (1, 1) {};
        \node [sdot, label = {left, gray!50}: {$(-\delta, -s\delta)$}] at ({-sqrt(2)}, {-sqrt(2)/2}) {};
        \draw [gray!50, dashed] ({sqrt(2)}, {sqrt(2)/2}) -- ({sqrt(2)}, 0);
        \draw [gray!50, dashed] (1, 1) -- (1, 0);
        \draw [gray!50, dashed] ({-sqrt(2)}, {-sqrt(2)/2}) -- ({-sqrt(2)}, 0);
        \draw [gray!50, dashed] (-1, -1) -- (-1, 0);
        \node [MediumPurple3!50] at (-0.4, 0.8) {$\varphi_\delta$};
        \node [DarkSlateGray3!60] at (0.4, -0.8) {$\psi_\delta$};
        \draw [MediumPurple3, decorate, decoration = {brace}] ({-sqrt(2)}, 0.1) to node [midway, above = {0.15cm}] {$S_1$} (0.98, 0.1);
        \draw [MediumPurple3, decorate, decoration = {brace}] (1.02, 0.1) to node [midway, above = {0.15cm}] {$S_2$} ({sqrt(2)}, 0.1);
        \draw [DarkSlateGray3, decorate, decoration = {brace}] (-1.02, -0.1) to node [midway, below = {0.15cm}] {$T_1$} ({-sqrt(2)}, -0.1);
        \draw [DarkSlateGray3, decorate, decoration = {brace}] ({sqrt(2)}, -0.1) to node [midway, below = {0.15cm}] {$T_2$} (-0.98, -0.1);
    \end{tikzpicture}
    \caption{An illustration of why $f(y) = sy$ is $2$-nice.} \label{fig:lines-nice}
\end{figure}

More generally, any function $f$ which defines a piece of an algebraic curve is $k$-nice for some constant $k$. To state this more formally, for a polynomial $g \in \RR[x, y]$, we write \[\cZ(g) = \{(x, y) \in \RR^2 \mid g(x, y) = 0\};\] an \deftext{algebraic curve} is a nonempty set $\gamma \subseteq \RR^2$ for which we can write $\gamma = \cZ(g)$ for some nonconstant $g \in \RR[x, y]$. Then if $\{(f(y), y) \mid y \in \RR\}$ is a subset of some algebraic curve $\cZ(g)$ (i.e., $g(f(y), y) = 0$ for all $y$), we can show that $f$ is $k$-nice for some $k$ only depending on $\deg g$. One way to see this is that given $\delta$, we can consider the set \[\cV_\delta = \{(x, y, z) \in \RR^3 \mid g(x, y) = 0, \, (z - x)^2 + y^2 = \delta^2\}.\] Then the graphs of $\varphi_\delta$ and $\psi_\delta$ are both subsets of the projection of $\cV_\delta$ onto the $yz$-plane. But $\cV_\delta$ is a $1$-dimensional variety, so by \cite[Lemma 4.12]{She22}, its projection onto the $yz$-plane is contained in a $1$-dimensional variety $\cZ(h)$ in $\RR^2$ for some $h \in \RR[x, y]$ whose degree is bounded in terms of $\deg g$. As shown in \cite[Section 2]{SZ24}, $\cZ(h)$ can be cut into a constant number (depending on $\deg h$) of monotone pieces; so the graphs of $\varphi_\delta$ and $\psi_\delta$ can also be cut into a constant number of monotone pieces. 

We conclude that in addition to strips around lines, Theorem \ref{thm:gen-curves} also applies to strips around algebraic curves satisfying the Lipschitz condition. (For an example, see Figure \ref{fig:polystrip-ex}.) 

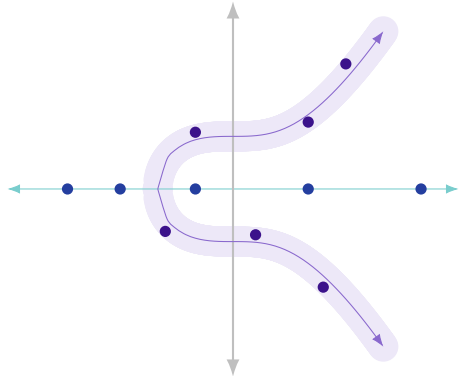
\begin{figure}[ht]
    \centering
    \begin{tikzpicture}[scale = 1]
        \foreach \x in {-1, -0.99, ..., 2} {
            \fill [MediumPurple3!15] (\x, {0.7*sqrt(\x*\x*\x + 1)}) circle (0.2);
            \fill [MediumPurple3!15] (\x, {-0.7*sqrt(\x*\x*\x + 1)}) circle (0.2);
        }

        \draw [DarkSlateGray3, Latex-Latex] (-3, 0) -- (3, 0);
        \draw [gray!50, thick, Latex-Latex] (0, -2.5) -- (0, 2.5);
        \draw [domain = -1:2, MediumPurple3, -Latex, smooth] plot (\x, {0.7*sqrt(\x*\x*\x + 1)});
        \draw [domain = -1:2, MediumPurple3, -Latex, smooth] plot (\x, {-0.7*sqrt(\x*\x*\x + 1)});
        \foreach \i in {1, 2.5, -0.5, -1.5, -2.2} {
            \node [sdot, myblue] at (\i, 0) {};
        }
        \foreach \i\j in {1.5/0.2, 1/-0.1, -0.5/0.1} {
            \node [sdot, mypurple] at (\i, {0.7*sqrt(\i*\i*\i + 1) + \j}) {};
        }
        \foreach \i\j in {-0.9/-0.2, 0.3/0.1, 1.2/-0.15} {
            \node [sdot, mypurple] at (\i, {-0.7*sqrt(\i*\i*\i + 1) + \j}) {};
        }
    \end{tikzpicture}
    \caption{For example, Theorem \ref{thm:gen-curves} applies if we take $\cP_2$ to lie in a strip around the curve $y^2 = x^3 + 1$.} \label{fig:polystrip-ex}
\end{figure}

\subsection{Overview} 

In the rest of the paper, we will prove Theorem \ref{thm:gen-curves}; since lines are $2$-nice, this directly implies Theorem \ref{thm:line-vs-strip}. (Theorem \ref{thm:gen-curves} does have an extra condition that the points in $\cP_2$ do not have repeated $x$-coordinates. However, in the setting of Theorem \ref{thm:line-vs-strip}, the spacing condition on $\cP_2$ together with the fact that $\cP_2$ lies on a non-vertical linear strip guarantees that at most a constant number of points in $\cP_2$ have any given $x$-coordinate, so we can ensure that no two points have the same $x$-coordinate by shrinking $\cP_2$ by a constant factor.)

First, we can assume without loss of generality that $s \geq 1$ (if a function is $s$-Lipschitz, it is also $\max\{s, 1\}$-Lipschitz). Also, if a set of points $\cP$ is $u$-spaced for some $u > 0$, then for any $u' > 0$, we can find a subset $\cP' \subseteq \cP$ consisting of a constant fraction (depending only on $u$ and $u'$) of $\cP$ which is $u'$-spaced; this allows us to fix a specific value of $u$ without loss of generality. So in order to prove Theorem \ref{thm:gen-curves}, it suffices to prove the following statement. 

\begin{lemma} \label{lem:cleaned-up-main}
    Fix $w > 0$, $s \geq 1$, and $k \in \NN$. Let $f : \RR \to \RR$ be $s$-Lipschitz and $k$-nice, and let \[\cP_1 \subseteq \{(x, 0) \mid x \in \RR\} \quad \text{and} \quad \cP_2 \subseteq \Strip_w(\{(f(y), y) \mid y \in \RR\})\] be $32ws$-spaced sets of points with $\abs{\cP_1} = m$ and $\abs{\cP_2} = n$, such that no two points in $\cP_2$ have the same $x$-coordinate. Then for every $\eps > 0$, we have \[\abs{\Delta(\cP_1, \cP_2)} = \Omega(\min\{m^{14/15}n^{8/15 - \eps}, m^{3/4}n^{3/4}, m^2, n^2\}),\] where the implicit constant depends only on $\eps$ and $k$. 
\end{lemma}

(The constant $32$ is not important; we did not attempt to optimize the numbers we use.)

First, in Section \ref{sec:main-ingredient}, we prove the following intermediate lemma, which is the key input needed to adapt Solymosi and Zahl's proximity argument (from \cite{SZ24}) to our setting. We denote the coordinates of a point $p \in \RR^2$ as $p_x$ and $p_y$. For every $\delta \in \Delta(\cP_1, \cP_2)$, we write \[\cR_\delta = \{(a, p) \in \cP_1 \times \cP_2 \mid \abs{ap} = \delta\}.\]  

\begin{lemma} \label{lem:prox-ingredient}
    Let $w$, $s$, $k$, $\cP_1$, and $\cP_2$ be as in Lemma \ref{lem:cleaned-up-main}. Then for each $\delta \in \Delta(\cP_1, \cP_2)$, we can find a list consisting of an $\Omega(1)$-fraction of the pairs $(a, p) \in \cR_\delta$ in which $p_y$ and $a_x$ are monotone (where the implicit constant depends on $k$). 
\end{lemma}

Then in Section \ref{sec:prox}, we use this to prove Lemma \ref{lem:cleaned-up-main}.

\section{The main ingredient for proximity} \label{sec:main-ingredient}

In this section, we prove Lemma \ref{lem:prox-ingredient}. First, we can write $\cR_\delta = \cR_\delta^+ \cup \cR_\delta^-$ where 
\begin{align*}
    \cR_\delta^+ &= \left\{(a, p) \in \cP_1 \times \cP_2 \mid a_x = p_x + \sqrt{\delta^2 - p_y^2}\right\}, \\
    \cR_\delta^- &= \left\{(a, p) \in \cP_1 \times \cP_2 \mid a_x = p_x - \sqrt{\delta^2 - p_y^2}\right\}.
\end{align*} 
We assume without loss of generality that $\sabs{\cR_\delta^+} \geq \sabs{\cR_\delta^-}$ and restrict our attention to just the pairs in $\cR_\delta^+$ (the proof when $\sabs{\cR_\delta^-} \geq \sabs{\cR_\delta^+}$ is essentially identical, with a few flipped signs). Note that for every $p \in \cP_2$, there is at most one $a \in \cP_1$ with $(a, p) \in \cR_\delta^+$.  

We classify the pairs $(a, p) \in \cR_\delta^+$ into three types:
\begin{itemize}
    \item We say $(a, p)$ is \deftext{short} if $\abs{p_y} \leq w$. 
    \item We say $(a, p)$ is \deftext{steep} if it is not short and $\abs{p_y} \geq 4s\sqrt{\delta^2 - p_y^2}$. 
    \item We say $(a, p)$ is \deftext{shallow} if it is neither short nor steep. 
\end{itemize} 
Intuitively, whether a pair $(a, p)$ is steep or shallow corresponds to how close $\protect\ora{pa}$ is to being vertical (specifically, how its slope compares to $-4s$), as depicted in Figure \ref{fig:steep-and-shallow}.
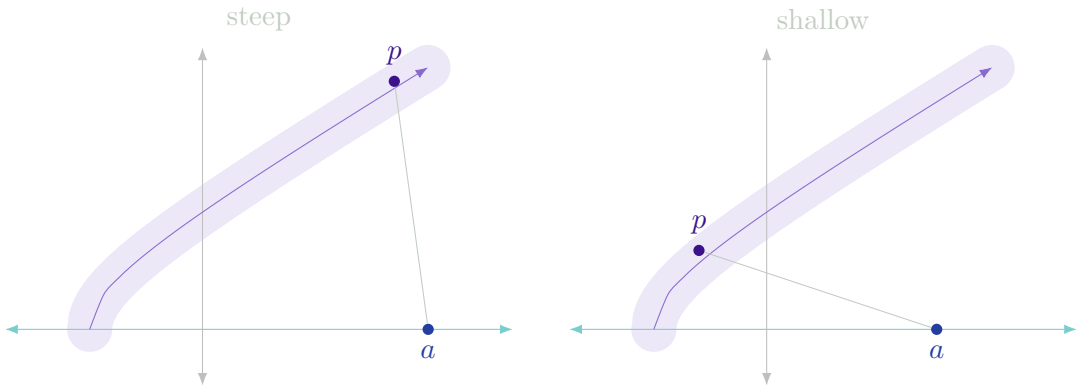
\begin{figure}[ht]
    \centering
    \begin{tikzpicture}[scale = 1.5]
        \foreach \x in {-1, -0.99, ..., 2} {
            \fill [MediumPurple3!15] (\x, {0.6*sqrt(\x*\x + 4*\x + 3)}) circle (0.2);
        }
        \draw [DarkSlateGray3, Latex-Latex] (-1.75, 0) -- (2.75, 0);
        \draw [gray!50, Latex-Latex] (0, -0.5) -- (0, 2.5);
        \draw [domain = -1:2, MediumPurple3, -Latex, smooth] plot (\x, {0.6*sqrt(\x*\x + 4*\x + 3)});
        \node [sdot, mypurple, label = {[mypurple] above: {$p$}}] (p1) at (1.7, 2.2) {};
        \node [sdot, myblue, label = {[myblue] below: {$a$}}] (a1) at (2, 0) {};
        \draw [Honeydew3] (p1) to (a1);
        \node [Honeydew3] at (0.5, 2.75) {steep};
        \begin{scope}[xshift = 5cm]
            \foreach \x in {-1, -0.99, ..., 2} {
                \fill [MediumPurple3!15] (\x, {0.6*sqrt(\x*\x + 4*\x + 3)}) circle (0.2);
            }
            \draw [DarkSlateGray3, Latex-Latex] (-1.75, 0) -- (2.75, 0);
            \draw [gray!50, Latex-Latex] (0, -0.5) -- (0, 2.5);
            \draw [domain = -1:2, MediumPurple3, -Latex, smooth] plot (\x, {0.6*sqrt(\x*\x + 4*\x + 3)});
            \node [sdot, mypurple, label = {[mypurple] above: {$p$}}] (p2) at (-0.6, 0.7) {};
            \node [sdot, myblue, label = {[myblue] below: {$a$}}] (a2) at ({-0.6 + sqrt(4.44)}, 0) {};
            \draw [Honeydew3] (p2) to (a2);
            \node [Honeydew3] at (0.5, 2.75) {shallow};
        \end{scope}
    \end{tikzpicture}
    \caption{A steep and shallow pair $(a, p)$.} \label{fig:steep-and-shallow}
\end{figure}

We will prove the following statements regarding each of these types:
\begin{enumerate}[(1)]
    \item \label{item:short} There is at most one short pair $(a, p)$. 
    \item \label{item:steep} We can find a list consisting of at least half the steep pairs in which $p_y$ and $a_x$ are monotone. 
    \item \label{item:shallow} We can find a list consisting of at least a $(1/k)$-fraction of the shallow pairs in which $p_y$ and $a_x$ are monotone. 
\end{enumerate}

This will imply Lemma \ref{lem:prox-ingredient}, since either short, steep, or shallow pairs have to account for at least a $(1/3)$-fraction of $\cR_\delta^+$, and therefore at least a $(1/6)$-fraction of $\cR_\delta$. 

\subsection{A preliminary observation}

First, the following observation converts the spacing condition on $\cP_2$ from a statement about the points having large pairwise \emph{distances} to a statement about the points having spaced-out \emph{$y$-coordinates}. This will be useful in several of the proofs. 

\begin{claim} \label{claim:y-spacing}
    Let $p, q \in \cP_2$ be distinct. Then $\sabs{p_y - q_y} \geq 16w$ and $\abs{p_x - q_x} \leq 2s\abs{p_y - q_y}$. 
\end{claim}

\begin{proof}
    First, because $\cP_2$ lies in the width-$w$ strip around the curve $\{(f(y), y) \mid y \in \RR\}$, we can find points $p^*$ and $q^*$ on this curve --- meaning that $p_x^* = f(p_y^*)$ and $q_x^* = f(q_y^*)$ --- with $\sabs{pp^*}, \sabs{qq^*} \leq w$. Since $f$ is $s$-Lipschitz, we have \[\sabs{p_x^* - q_x^*} = \sabs{f(p_y^*) - f(q_y^*)} \leq s\sabs{p_y^* - q_y^*}.\] Then since $\sabs{pp^*}, \sabs{qq^*} \leq w$, we have 
    \begin{align*}
        \sabs{p_x - q_x} &\leq \sabs{p_x - p_x^*} + \sabs{q_x - q_x^*} + \sabs{p_x^* - q_x^*} \\
        &\leq 2w + s\sabs{p_y^* - q_y^*} \\
        &\leq 2w + s(\sabs{p_y - p_y^*} + \sabs{q_y - q_y^*} + \sabs{p_y - q_y}) \\
        &\leq 2w + 2ws + s\sabs{p_y - q_y}. 
    \end{align*}
    Since $s \geq 1$, this means 
    \begin{equation}
        \sabs{p_x - q_x} \leq 4ws + s\sabs{p_y - q_y}. \label{eqn:xspace-to-yspace}
    \end{equation}

    Assume for contradiction that $\sabs{p_y - q_y} < 16w$. Then $\sabs{p_x - q_x} < 4ws + 16ws \leq 20ws$, so we have \[\abs{pq} = \sqrt{(p_x - q_x)^2 + (p_y - q_y)^2} \leq \sqrt{(20ws)^2 + (16w)^2} \leq \sqrt{20^2 + 16^2} \cdot ws < 32ws,\] contradicting the assumption that $\cP_2$ is $32ws$-spaced.  

    So $\sabs{p_y - q_y} \geq 16w$, which proves the first statement of Claim \ref{claim:y-spacing}. The second statement then follows from plugging $4w \leq \sabs{p_y - q_y}$ into \eqref{eqn:xspace-to-yspace}, which gives \[\sabs{p_x - q_x} \leq 4ws + s\sabs{p_y - q_y} \leq 2s\sabs{p_y - q_y}.\qedhere\] 
\end{proof}

Note that Claim \ref{claim:y-spacing} immediately implies \ref{item:short}, the statement that there is at most one short pair. Claim \ref{claim:y-spacing} also means that all pairs $(a, p) \in \cR_\delta^+$ have different values of $p_y$. 

It now remains to prove \ref{item:steep} and \ref{item:shallow}, which we do in the next subsections. 

\subsection{Steep pairs} 

In this subsection, we prove the following more specific version of \ref{item:steep}. 

\begin{lemma} \label{lem:steep}
    Suppose that $(a, p)$ and $(b, q)$ are steep pairs such that $p_y$ and $q_y$ have the same sign and $\abs{p_y} < \abs{q_y}$. Then $a_x > b_x$. 
\end{lemma}

Lemma \ref{lem:steep} implies \ref{item:steep} because it means that steep pairs with $p_y > 0$ form a list with increasing $p_y$ and decreasing $a_x$, and steep pairs with $p_y < 0$ form a list with increasing $p_y$ and increasing $a_x$; one of these lists accounts for at least half the steep pairs.

The geometric intuition behind the proof of Lemma \ref{lem:steep} is that when $(a, p)$ is steep, changing the $y$-coordinate of $p$ by a little has a huge effect on $a$ (the corresponding point on the $x$-axis with $\abs{ap} = \delta$). The Lipschitz condition (or more precisely, the second part of Claim \ref{claim:y-spacing}) means that the difference in $x$-coordinates between $p$ and $q$ is controlled in terms of the difference in $y$-coordinates. Together, these mean that the change in $y$-coordinates as we go from $p$ to $q$ has much greater effect on $a_x - b_x$ than the change in $x$-coordinates. This is illustrated in Figure \ref{fig:steep-proof}. 

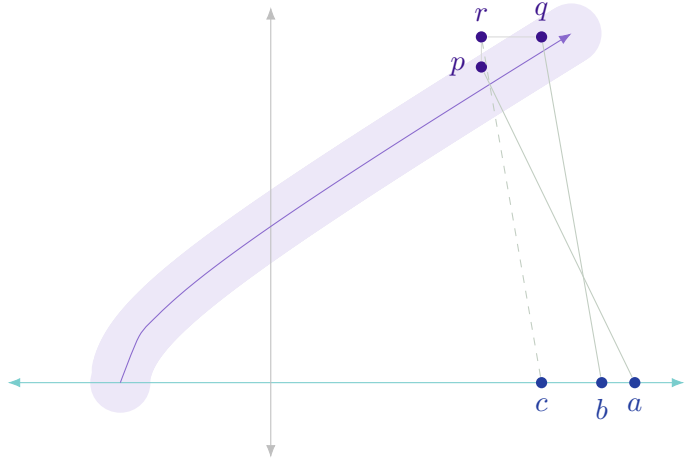
\begin{figure}[ht]
    \begin{tikzpicture}[scale = 2]
        \foreach \x in {-1, -0.99, ..., 2} {
            \fill [MediumPurple3!15] (\x, {0.6*sqrt(\x*\x + 4*\x + 3)}) circle (0.2);
        }
        \draw [DarkSlateGray3, Latex-Latex] (-1.75, 0) -- (2.75, 0);
        \draw [gray!50, Latex-Latex] (0, -0.5) -- (0, 2.5);
        \draw [domain = -1:2, MediumPurple3, -Latex, smooth] plot (\x, {0.6*sqrt(\x*\x + 4*\x + 3)});
        \node [sdot, mypurple, label = {[mypurple] above: {$q$}}] (q) at (1.8, 2.3) {};
        \node [sdot, myblue, label = {[myblue] below: {$b$}}] (b) at (2.2, 0) {};
        \node [sdot, mypurple, label = {[mypurple] left: {$p$}}] (p) at (1.4, 2.1) {};
        \node [sdot, myblue, label = {[myblue] below: {$a$}}] (a) at ({1.4 + sqrt(2.3^2 + 0.4^2 - 2.1^2)}, 0) {};
        \node [sdot, mypurple, label = {[mypurple] above: {$r$}}] (r) at (1.4, 2.3) {};
        \node [sdot, myblue, label = {[myblue] below: {$c$}}] (c) at (1.8, 0) {};
        \draw [Honeydew3] (a) -- (p);
        \draw [Honeydew3] (b) -- (q);
        \draw [Honeydew3, dashed] (c) -- (r);
        \draw [gray!30] (p) -- (r) -- (q);
    \end{tikzpicture}
    \caption{If we move $p$ to $r$ and then to $q$ while keeping track of the corresponding point on the $x$-axis, the change in $p_y$ (when moving from $p$ to $r$) has more effect on this point than the change in $p_x$ (when moving from $r$ to $q$).} \label{fig:steep-proof}
\end{figure}

\begin{proof}
    We have $a_x = p_x + \sqrt{\delta^2 - p_y^2}$ and $b_x = q_x + \sqrt{\delta^2 - q_y^2}$, so 
    \begin{equation}
        a_x - b_x = (p_x - q_x) + \left(\sqrt{\delta^2 - p_y^2} - \sqrt{\delta^2 - q_y^2}\right). \label{eqn:ax-bx}
    \end{equation} 
    The main idea is to show that the first term on the right-hand side of \eqref{eqn:ax-bx} is negligible compared to the second (which is always positive, as $\abs{p_y} < \abs{q_y}$). To do so, first note that 
    \begin{equation}
        \abs{p_x - q_x} \leq 2s\sabs{p_y - q_y} \label{eqn:x-diff}
    \end{equation} 
    by Claim \ref{claim:y-spacing}. Meanwhile, we can rewrite the rightmost term of \eqref{eqn:ax-bx} as 
    \begin{equation}
        \sqrt{\delta^2 - p_y^2} - \sqrt{\delta^2 - q_y^2} = \frac{q_y^2 - p_y^2}{\sqrt{\delta^2 - p_y^2} + \sqrt{\delta^2 - q_y^2}} = \frac{(\abs{p_y} + \abs{q_y})}{\sqrt{\delta^2 - p_y^2} + \sqrt{\delta^2 - q_y^2}} \cdot \abs{p_y - q_y}. \label{eqn:second-term}
    \end{equation} 
    To deal with the factor in front of $\abs{p_y - q_y}$, note that for any $\alpha_1, \alpha_2, \beta_1, \beta_2 \geq 0$, we have \[\frac{\alpha_1 + \alpha_2}{\beta_1 + \beta_2} = \frac{\beta_1}{\beta_1 + \beta_2} \cdot \frac{\alpha_1}{\beta_1} + \frac{\beta_2}{\beta_1 + \beta_2} \cdot \frac{\alpha_2}{\beta_2} \geq \min\left\{\frac{\alpha_1}{\beta_1}, \frac{\alpha_2}{\beta_2}\right\}\] (where if one of $\alpha_1/\beta_1$ and $\alpha_2/\beta_2$ has denominator $0$, we treat its value as $+\infty$). Combining this fact with the assumption that $(a, p)$ and $(b, q)$ are steep gives \[\frac{(\abs{p_y} + \abs{q_y})}{\sqrt{\delta^2 - p_y^2} + \sqrt{\delta^2 - q_y^2}} \geq \min\left\{\frac{\abs{p_y}}{\sqrt{\delta^2 - p_y^2}}, \frac{\abs{q_y}}{\sqrt{\delta^2 - q_y^2}}\right\} \geq 4s.\] Plugging this into \eqref{eqn:second-term} gives 
    \begin{equation}
        \sqrt{\delta^2 - p_y^2} - \sqrt{\delta^2 - q_y^2} \geq 4s\abs{p_y - q_y}.\label{eqn:second-term-final}
    \end{equation} 
    Finally, plugging \eqref{eqn:x-diff} and \eqref{eqn:second-term-final} into \eqref{eqn:ax-bx}, we get \[a_x - b_x \geq -2s\sabs{p_y - q_y} + 4s\sabs{p_y - q_y} > 0.\qedhere\] 
\end{proof}

\subsection{Shallow pairs} 

In this subsection, we prove \ref{item:shallow}. First, for each $p \in \cP_2$, we define $p^*$ as some point on the curve $\{(f(y), y) \mid y \in \RR\}$ with $\sabs{pp^*} \leq w$. We will then prove the following statement. 

\begin{lemma} \label{lem:shallow}
    Suppose that $\varphi_\delta$ is monotone on a set $S \subseteq \RR$. Consider the list of all shallow pairs $(a, p)$ with $p_y^* \in S$, sorted in increasing order of $p_y$. Then in this list, $a_x$ is monotone.
\end{lemma}

Lemma \ref{lem:shallow} implies \ref{item:shallow} because the assumption that $f$ is $k$-nice means that we can partition $\RR$ into $k$ such sets $S$, which gives a partition of our shallow pairs into $k$ lists in which both $p_y$ and $a_x$ are monotone. 

\begin{proof}
    Note that if $p_y < q_y$, then $p_y^* < q_y^*$ because of Claim \ref{claim:y-spacing}. Indeed, Claim \ref{claim:y-spacing} means that $q_y - p_y \geq 16w$, so $q_y^* - p_y^* \geq q_y - p_y - 2w \geq 16w - 2w > 0$. This means that sorting our pairs $(a, p)$ by their value of $p_y$ is equivalent to sorting them by their value of $p_y^*$. Then the assumption that $\varphi_\delta$ is monotone on $S$ (and that $p_y^* \in S$ for all these pairs) means that $\varphi_\delta(p_y^*)$ is monotone on our sorted list. Our goal is to use this to show that $a_x$ is also monotone; for this, we will use the following claim. 

    \begin{claim} \label{claim:shallow-wiggle}
        If $(a, p)$ is shallow, then $\sabs{a_x - \varphi_\delta(p_y^*)} \leq 13ws$. 
    \end{claim}

    The geometric intuition behind Claim \ref{claim:shallow-wiggle} is that $a$ is the point on the $x$-axis at a distance $\delta$ from $p$ (and to the right of $p$), while $(\varphi_\delta(p_y^*), 0)$ is the analogous point for $p^*$ (by the definition of $\varphi_\delta$). And the fact that $(a, p)$ is shallow means that moving $p$ by a small distance only has a small effect on the corresponding point $a$. This is illustrated in Figure \ref{fig:shallow-proof}. 

    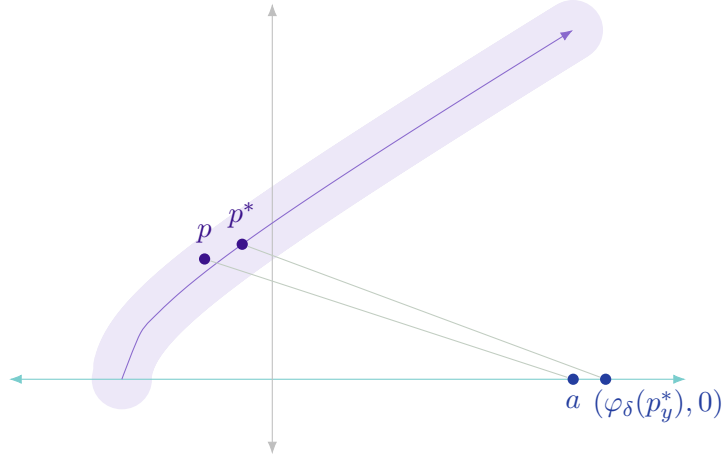
\begin{figure}[ht]
        \begin{tikzpicture}[scale = 2]
            \foreach \x in {-1, -0.99, ..., 2} {
                \fill [MediumPurple3!15] (\x, {0.6*sqrt(\x*\x + 4*\x + 3)}) circle (0.2);
            }
            \draw [DarkSlateGray3, Latex-Latex] (-1.75, 0) -- (2.75, 0);
            \draw [gray!50, Latex-Latex] (0, -0.5) -- (0, 2.5);
            \draw [domain = -1:2, MediumPurple3, -Latex, smooth] plot (\x, {0.6*sqrt(\x*\x + 4*\x + 3)});
            \node [sdot, mypurple, label = {[mypurple] above: {$p$}}] (p) at (-0.45, 0.8) {};
            \node [sdot, mypurple, label = {[mypurple] above: {$p^*$}}] (p2) at (-0.2, {0.6*sqrt(0.04 - 4*0.2 + 3)}) {};
            \node [sdot, myblue, label = {[myblue] below: {$a$}}] (a) at (2, 0) {};
            \node [sdot, myblue] (a2) at ({-0.2 + sqrt(2.45^2 + 0.8^2 - 0.36*2.24)}, 0) {};
            \node [myblue, anchor = north] at ($(a2) + (0.35, 0)$) {$(\varphi_\delta(p_y^*), 0)$} {};
            \draw [Honeydew3] (a) -- (p);
            \draw [Honeydew3] (p2) -- (a2);
        \end{tikzpicture}
        \caption{Moving $p$ to $p^*$ only has a small effect on the point on the $x$-axis at a distance $\delta$, which moves from $a$ to $(\varphi_\delta(p_y^*), 0)$.} \label{fig:shallow-proof}
    \end{figure}

    \begin{proof}
        First, we have $a_x = p_x + \sqrt{\delta^2 - p_y^2}$ and \[\varphi_\delta(p_y^*) = f(p_y^*) + \sqrt{\delta^2 - (p_y^*)^2} = p_x^* + \sqrt{\delta^2 - (p_y^*)^2},\] so by the triangle inequality, \[\abs{a_x - \varphi_\delta(p_y^*)} \leq \sabs{p_x - p_x^*} + \abs{\sqrt{\delta^2 - p_y^2} - \sqrt{\delta^2 - (p_y^*)^2}}.\] For the first term, we have $\sabs{p_x - p_x^*} \leq \sabs{pp^*} \leq w$. For the second term, we can first write\[\abs{\sqrt{\delta^2 - p_y^2} - \sqrt{\delta^2 - (p_y^*)^2}} = \frac{\sabs{(p_y^*)^2 - p_y^2}}{\sqrt{\delta^2 - p_y^2} + \sqrt{\delta^2 - (p_y^*)^2}} \leq \frac{\sabs{p_y^* - p_y}\sabs{p_y^* + p_y}}{\sqrt{\delta^2 - p_y^2}}.\] We have $\sabs{p_y^* - p_y} \leq \sabs{pp^*} \leq w$ and $\sabs{p_y^* + p_y} \leq 2\sabs{p_y} + w \leq 3\sabs{p_y}$ (here we are using the fact that $(a, p)$ is not short, so $\sabs{p_y} \geq w$). This means \[\abs{\sqrt{\delta^2 - p_y^2} - \sqrt{\delta^2 - (p_y^*)^2}} \leq \frac{3w\abs{p_y}}{\sqrt{\delta^2 - p_y^2}}.\] And the assumption that $(a, p)$ is shallow means that this is less than $12ws$. So we get \[\sabs{a_x - \varphi_\delta(p_y^*)} \leq w + 12ws \leq 13ws.\qedhere\] 
    \end{proof}

    To complete the proof of Lemma \ref{lem:shallow}, note that Claim \ref{claim:shallow-wiggle} together with the assumption that $\cP_1$ is $32ws$-spaced means that if $(a, p)$ and $(b, q)$ are shallow and $a_x < b_x$, then $\varphi_\delta(p_y^*) < \varphi_\delta(q_y^*)$. And we saw earlier that $\varphi_\delta(p_y^*)$ is monotone in our list, so $a_x$ must be as well. 
\end{proof}

We have now proved each of \ref{item:short}, \ref{item:steep}, and \ref{item:shallow}, completing the proof of Lemma \ref{lem:prox-ingredient}. 

\section{The proximity argument} \label{sec:prox}

In this section, we prove Lemma \ref{lem:cleaned-up-main} by running an adaptation of Solymosi and Zahl's proximity argument (sketched in Subsection \ref{subsec:background}). We define the \deftext{distance energy} between $\cP_1$ and $\cP_2$ as 
\begin{equation}
    \sfE(\cP_1, \cP_2) = \abs{\{(a, b, p, q) \in \cP_1^2 \times \cP_2^2 \mid \abs{ap} = \abs{bq}\}}. \label{eqn:dist-energy}
\end{equation} 
As described in Subsection \ref{subsec:background}, in order to prove a lower bound on $\sabs{\Delta(\cP_1, \cP_2)}$, it suffices to prove an \emph{upper} bound on $\sfE(\cP_1, \cP_2)$. Specifically, for each $\delta \in \Delta(\cP_1, \cP_2)$, we can define $\cR_\delta = \{(a, p) \in \cP_1 \times \cP_2 \mid \abs{ap} = \delta\}$ and $r_\delta = \sabs{\cR_\delta}$; then by the Cauchy--Schwarz inequality, we have
\begin{equation}
    \sfE(\cP_1, \cP_2) = \sum_{\delta \in \Delta(\cP_1, \cP_2)} r_\delta^2 \geq \frac{(\sum_{\delta \in \Delta(\cP_1, \cP_2)} r_\delta)^2}{\abs{\Delta(\cP_1, \cP_2)}} = \frac{m^2n^2}{\abs{\Delta(\cP_1, \cP_2)}}. \label{eqn:cauchy--schwarz}
\end{equation}
In particular, Bruner and Sharir \cite{BS18} proved their bound \eqref{eqn:bs18} for the setting where $\cP_1$ lies on a line and $\cP_2$ is unrestricted by using incidence bounds to upper-bound $\sfE(\cP_1, \cP_2)$. 

We prove Lemma \ref{lem:cleaned-up-main} by incorporating proximity into this argument. The main idea of Solymosi and Zahl's proximity argument is to restrict the distance energy (as defined in \eqref{eqn:dist-energy}) to only consider quadruples $(a, b, p, q)$ where $a$ is `close' to $b$ and $p$ is `close' to $q$. 

To formally define the appropriate notion of `closeness' for our setting, imagine that we sort the points in $\cP_1$ by their $x$-coordinate; for $a \in \cP_1$, we use $i_1(a)$ to denote the index of $a$ under this sorting. (For example, if $\cP_1 = \{(-1, 0), (2, 0), (4, 0), (5, 0)\}$ and $a = (4, 0)$, then $i_1(a) = 3$.) We sort the points in $\cP_2$ by their $y$-coordinate; for $p \in \cP_2$, we use $i_2(p)$ to denote the index of $p$ under this sorting. (By Claim \ref{claim:y-spacing}, all points in $\cP_2$ have distinct $y$-coordinates.)

For $t \in [0, 1]$, we say a pair $(a, b) \in \cP_1^2$ is \deftext{$t$-close} if $\sabs{i_1(a) - i_1(b)} \leq tm$; similarly, we say a pair $(p, q) \in \cP_2^2$ is $t$-close if $\sabs{i_2(p) - i_2(q)} \leq tn$. We define \[\sfE_t(\cP_1, \cP_2) = \abs{\{(a, b, p, q) \in \cP_1^2 \times \cP_2^2 \mid \abs{ap} = \abs{bq}, \, \text{$(a, b)$ and $(p, q)$ are $t$-close}\}}.\] 

In Subsection \ref{subsec:lower}, we prove a lower bound on $\sfE_t(\cP_1, \cP_2)$ in terms of $\sfE(\cP_1, \cP_2)$; in Subsection \ref{subsec:upper}, we prove an upper bound on $\sfE(\cP_1, \cP_2)$ using incidence bounds; and in Subsection \ref{subsec:computations}, we combine these bounds and choose an appropriate value of $t$ to complete the proof. 

\subsection{A lower bound} \label{subsec:lower}

In this subsection, we prove the following lower bound on $\sfE_t(\cP_1, \cP_2)$. 

\begin{lemma} \label{lem:prox-lower}
    We have $\sfE_t(\cP_1, \cP_2) = \Omega(t \cdot \sfE(\cP_1, \cP_2))$ (where the implicit constant depends on $k$). 
\end{lemma}

Intuitively, the bound of Lemma \ref{lem:prox-lower} is useful because in the definition of $\sfE_t(\cP_1, \cP_2)$, we place proximity restrictions on both $(a, b)$ and $(p, q)$. Each of these restrictions shrinks the number of `allowed' pairs by a factor of roughly $t$. So if proximity had no relation to the condition $\abs{ap} = \abs{bq}$, then we would expect these restrictions to shrink the number of quadruples $(a, b, p, q)$ by a factor of roughly $t^2$. The factor-of-$t$ shrinkage given by Lemma \ref{lem:prox-lower} is much better than this.

We will prove the following more specific statement. 

\begin{lemma} \label{lem:prox-lower-specific}
    For every $\delta \in \Delta(\cP_1, \cP_2)$, we have 
    \begin{equation}
        \abs{\{(a, b, p, q) \in \cP_1^2 \times \cP_2^2 \mid \abs{ap} = \abs{bq} = \delta, \, \text{$(a, b)$ and $(p, q)$ are $t$-close}\}} = \Omega(tr_\delta^2). \label{eqn:prox-lower-specific}
    \end{equation}
\end{lemma}

Lemma \ref{lem:prox-lower} follows from Lemma \ref{lem:prox-lower-specific} by summing over all $\delta \in \Delta(\cP_1, \cP_2)$ --- the left-hand side of \eqref{eqn:prox-lower-specific} sums to $\sfE_t(\cP_1, \cP_2)$, while on the right-hand side, we have $\sfE(\cP_1, \cP_2) = \sum_{\delta \in \Delta(\cP_1, \cP_2)} r_\delta^2$. 

\begin{proof}
    First, by Lemma \ref{lem:prox-ingredient}, we can find a list of $\ell = \Omega(r_\delta)$ pairs $(a, p) \in \cR_\delta$ in which $a_x$ and $p_y$ are both monotone, meaning that $i_1(a)$ and $i_2(p)$ are both monotone. Let this list be $(a_1, p_1)$, $(a_2, p_2)$, \ldots, $(a_\ell, p_\ell)$. We assume without loss of generality that $i_1(a)$ and $i_2(p)$ are both \emph{increasing} in this list; the proof when one is decreasing is essentially identical. 

    Let $s = \sceil{t\ell/8} - 1$ (so that $s \leq t\ell/8$, but $s + 1 \geq t\ell/8$). The main idea is to show that for most indices $1 \leq j \leq \ell - s$, both $(a_j, a_{j + s})$ and $(p_j, p_{j + s})$ are $t$-close. 
    
    First we consider $(a_j, a_{j + s})$. We have $i_1(a_{j + s}) - i_1(a_j) \geq 0$ for all $j$, and \[\sum_{j = 1}^{\ell - s} (i_1(a_{j + s}) - i_1(a_j)) = \sum_{j = \ell - s + 1}^\ell i_1(a_j) - \sum_{j = 1}^s i_1(a_j) \leq sm \leq \frac{t\ell m}{8}\] (since the initial sum telescopes). So the number of indices $j$ for which $(a_j, a_{j + s})$ is \emph{not} $t$-close, meaning that $i_1(a_{j + s}) - i_1(a_j) > tm$, is at most $\ell/8$. Similarly for $(p_j, p_{j + s})$, we have $i_2(p_{j + s}) - i_2(p_j) \geq 0$ for all $j$, and \[\sum_{j = 1}^{\ell - s} (i_2(p_{j + s}) - i_2(p_j)) = \sum_{j = \ell - s + 1}^\ell i_2(p_j) - \sum_{j = 1}^s i_2(p_j) \leq sn \leq \frac{t\ell n}{8},\] which means the number of indices $j$ for which $(p_j, p_{j + s})$ is not $t$-close is at most $\ell/8$. This means there are at least \[\ell - s - \frac{\ell}{8} - \frac{\ell}{8} \geq \frac{\ell}{2}\] indices $j$ for which both $(a_j, a_{j + s})$ and $(p_j, p_{j + s})$ are both $t$-close. 

    Now suppose that we take $j$ to be any such index and take any $j'$ with $j \leq j' \leq j + s$. Then the pairs $(a_j, a_{j'})$ and $(p_j, p_{j'})$ are also $t$-close, so the quadruple $(a_j, a_{j'}, p_j, p_{j'})$ is among those counted by the left-hand side of \eqref{eqn:prox-lower-specific}. There are at least $\ell/2$ choices for $j$ and $s + 1 \geq t\ell/8$ choices for $j'$, so this means \[\abs{\{(a, b, p, q) \in \cP_1^2 \times \cP_2^2 \mid \abs{ap} = \abs{bq} = \delta, \, \text{$(a, b)$ and $(p, q)$ are $t$-close}\}} \geq \frac{t\ell^2}{16} = \Omega(tr_\delta^2)\] (since our application of Lemma \ref{lem:prox-ingredient} guaranteed that $\ell = \Omega(r_\delta)$). 
\end{proof}

\subsection{An upper bound} \label{subsec:upper}

In this subsection, we prove the following upper bound on $\sfE_t(\cP_1, \cP_2)$. 

\begin{lemma} \label{lem:upper}
    As long as $tm, tn \geq 1$, for every $\eta > 0$, we have 
    \begin{equation}
        \sfE_t(\cP_1, \cP_2) - 4mn = O(t^{15/11}m^{12/11}n^{18/11 + \eta} + t^{4/3}m^{4/3}n^{4/3} + tm^2 + tn^2).\label{eqn:upper}
    \end{equation}
\end{lemma}

We will prove Lemma \ref{lem:upper} using an incidence bound for algebraic curves due to Sharir and Zahl \cite{SZ17}. We first give a few definitions needed to state this bound: 
\begin{itemize}
    \item For a set of points $\Pi$ and a set of curves $\Gamma$ in $\RR^2$, an \deftext{incidence} between $\Pi$ and $\Gamma$ is a pair $(p, \gamma) \in \Pi \times \Gamma$ where $p$ lies on $\gamma$. We denote the number of incidences between $\Pi$ and $\Gamma$ by \[\cI(\Pi, \Gamma) = \abs{\{(p, \gamma) \in \Pi \times \Gamma \mid p \in \gamma\}}.\] 
    \item A polynomial $g \in \RR[x, y]$ is \deftext{irreducible} if it is not possible to write $g = g_1g_2$ for nonconstant polynomials $g_1, g_2 \in \RR[x, y]$. 
    \item As in Subsection \ref{subsec:nonlinear-strips}, we write $\cZ(g) = \{(x, y) \in \RR^2 \mid g(x, y) = 0\}$. An \deftext{algebraic curve} is a nonempty set $\gamma \subseteq \RR^2$ such that $\gamma = \cZ(g)$ for some nonconstant $g \in \RR[x, y]$. An algebraic curve $\gamma$ is \deftext{irreducible} if we can write $\gamma = \cZ(g)$ for some irreducible $g$.
    \item We say a set of curves $\cC$ is a \deftext{$3$-parameter family} if all the curves in $\cC$ are of the form $\cZ(g)$ for polynomials $g \in \RR[x, y]$ whose coefficients are themselves polynomials in $3$ parameters. More formally, this means there is a polynomial $G \in \RR[x, y, \alpha_1, \alpha_2, \alpha_3]$ such that letting $g_{\alpha_1, \alpha_2, \alpha_3} \in \RR[x, y]$ be the polynomial given by \[g_{\alpha_1, \alpha_2, \alpha_3}(x, y) = G(x, y, \alpha_1, \alpha_2, \alpha_3)\] (for any fixed $\alpha_1, \alpha_2, \alpha_3 \in \RR$), we have \[\cC \subseteq \{\cZ(g_{\alpha_1, \alpha_2, \alpha_3}) \mid \alpha_1, \alpha_2, \alpha_3 \in \RR\}.\] We define the \deftext{degree} of $\cC$ as $\deg G$ (more precisely, the minimum value of $\deg G$ over all $G$ which could be used to define $\cC$). For example, the collection of circles in $\RR^2$ is a $3$-parameter family of degree $2$, corresponding to \[G(x, y, \alpha_1, \alpha_2, \alpha_3) = (x - \alpha_1)^2 + (y - \alpha_2)^2 - \alpha_3^2.\]
\end{itemize}
Then the statement of the bound is as follows. (Sharir and Zahl work in a more general setting, but the setting we have described here is easier to define and is enough for our purposes.)

\begin{theorem}[Sharir--Zahl] \label{thm:sz17}
    Let $\Pi$ be a set of points in $\RR^2$, and let $\Gamma$ be a set of irreducible algebraic curves in $\RR^2$ from a $3$-parameter family of bounded degree. Then for all $\eta > 0$, we have \[\cI(\Pi, \Gamma) = O(\abs{\Pi}^{6/11}\abs{\Gamma}^{9/11 + \eta} + \abs{\Pi}^{2/3}\abs{\Gamma}^{2/3} + \abs{\Pi} + \abs{\Gamma}).\] 
\end{theorem}

\begin{proof}[Proof of Lemma \ref{lem:upper}]
    We first separately account for the contribution to $\sfE_t(\cP_1, \cP_2)$ from quadruples $(a, b, p, q)$ with $p_y = \pm q_y$. By Claim \ref{claim:y-spacing}, there is at most one point in $\cP_2$ with any given $y$-coordinate. So quadruples $(a, b, p, q)$ with $p_y = \pm q_y$ have a total contribution of at most $4mn$ to $\sfE_t(\cP_1, \cP_2)$: There are $n$ ways to choose $p$ and $m$ ways to choose $a$; then there are at most $2$ ways to choose $q$ such that $q_y = \pm p_y$; and finally, there are at most $2$ ways to choose $b$ such that $\abs{ap} = \abs{bq}$. 

    We will now bound the contribution of the remaining pairs $(a, b, p, q)$ by the right-hand side of \eqref{eqn:upper}, using the Sharir--Zahl incidence bound (Theorem \ref{thm:sz17}). First, we can write the condition $\abs{ap} = \abs{bq}$ as \[(a_x - p_x)^2 + p_y^2 = (b_x - q_x)^2 + q_y^2.\] Then we can set up an incidence problem in $\RR^2$ by using each possible pair $(a, b)$ to define a point and each $(p, q)$ to define a curve: We define the set of points \[\Pi = \{(a_x, b_x) \mid \text{$(a, b) \in \cP_1^2$ is $t$-close}\}.\] We define a set of curves $\Gamma$ as follows: For each $t$-close pair $(p, q) \in \cP_2^2$ with $p_y \neq \pm q_y$, we include the curve defined by \[(x - p_x)^2 + p_y^2 = (y - q_x)^2 + q_y^2,\] where $x$ and $y$ are the variables used to define the curve, and $p_x$, $p_y$, $q_x$, and $q_y$ are constants. Then every pair $(a, b, p, q)$ with $p_y \neq \pm q_y$ which contributes to $\sfE_t(\cP_1, \cP_2)$ corresponds to an incidence between $\Pi$ and $\Gamma$. 

    Note that the condition $p_y \neq \pm q_y$ ensures that all curves in $\Gamma$ are irreducible. Also, the curves corresponding to different pairs $(p, q)$ are distinct: Given some curve $\gamma \in \Gamma$, we can recover $p_x$ from the coefficient of $x$ and $q_x$ from the coefficient of $y$, and since no two points in $\cP_2$ have the same $x$-coordinate by assumption, this means we can recover $p$ and $q$. 

    Finally, $\abs{\Pi}$ is the number of $t$-close pairs $(a, b) \in \cP_1^2$, which is at most $3tm^2$ (there are $m$ choices for $a$, and each corresponds to at most $2tm + 1 \leq 3tm$ choices for $b$). Similarly, $\abs{\Gamma}$ is at most the number of $t$-close pairs $(p, q) \in \cP_2^2$, which is at most $3tn^2$. So applying Theorem \ref{thm:sz17} to $\Pi$ and $\Gamma$ gives that for every $\eta > 0$ we have \[\cI(\Pi, \Gamma) = O(t^{15/11}m^{12/11}n^{18/11 + \eta} + t^{4/3}m^{4/3}n^{4/3} + tm^2 + tn^2).\qedhere\]
\end{proof}

\subsection{Final computations} \label{subsec:computations}

In this subsection, we combine Lemmas \ref{lem:prox-lower} and \ref{lem:upper} to prove the following upper bound on $\sfE(\cP_1, \cP_2)$.

\begin{lemma} \label{lem:final}
    For every $\eps > 0$, we have \[\sfE(\cP_1, \cP_2) = O(m^{16/15}n^{22/15 + \eps} + m^{5/4}n^{5/4} + n^2 + m^2).\]  
\end{lemma}

\begin{proof}
    We first choose a value of $t$. To do so, let $c$ be the implicit constant in Lemma \ref{lem:prox-lower}, so that Lemma \ref{lem:prox-lower} gives $\sfE_t(\cP_1, \cP_2) \geq ct\cdot \sfE(\cP_1, \cP_2)$; we can assume without loss of generality that $c \leq 1$. We then set \[t = \max\left\{\frac{8mn}{c \cdot \sfE(\cP_1, \cP_2)}, \frac{1}{m}\right\}.\] We can check that this value of $t$ is `reasonable' in the following ways:
    \begin{itemize}
        \item If $t > 1$, then we immediately get $\sfE(\cP_1, \cP_2) = O(mn) = O(m^2 + n^2)$, and we are done. So we can assume $t \leq 1$.
        \item We have $tm \geq 1$ by definition. Meanwhile, we must have $\sfE(\cP_1, \cP_2) \leq 2mn^2$ --- if we want to choose a quadruple $(a, b, p, q) \in \cP_1^2 \times \cP_2^2$ with $\abs{ap} = \abs{bq}$, there are $n$ choices for each of $p$ and $q$ and $m$ choices for $a$; then there are at most $2$ choices for $b$ (as $b$ must lie on the $x$-axis and be a specified distance from $q$). This ensures \[tn \geq \frac{8mn^2}{c \cdot \sfE(\cP_1, \cP_2)} \geq 1.\] 
        \item Finally, by Lemma \ref{lem:prox-lower} we have $\sfE_t(\cP_1, \cP_2) \geq ct \cdot \sfE(\cP_1, \cP_2) \geq 8mn$, so \[\sfE_t(\cP_1, \cP_2) - 4mn \geq \frac{\sfE_t(\cP_1, \cP_2)}{2}.\] (We need this in order for the bound on $\sfE_t(\cP_1, \cP_2) - 4mn$ from Lemma \ref{lem:upper} to be useful.)
    \end{itemize}
    Now set $\eta = \eps/2$. Then combining the lower and upper bounds on $\sfE_t(\cP_1, \cP_2) - 4mn$ from Lemmas \ref{lem:prox-lower} and \ref{lem:upper}, we get \[t \cdot \sfE(\cP_1, \cP_2) = O(t^{15/11}m^{12/11}n^{18/11 + \eta} + t^{4/3}m^{4/3}n^{4/3} + tm^2 + tn^2),\] and dividing by $t$ gives \[\sfE(\cP_1, \cP_2) = O(t^{4/11}m^{12/11}n^{18/11 + \eta} + t^{1/3}m^{4/3}n^{4/3} + m^2 + n^2).\] We now perform casework on which of the four terms on the right-hand side is largest. If the third or fourth terms are largest, then we get $\sfE(\cP_1, \cP_2) = O(m^2)$ or $\sfE(\cP_1, \cP_2) = O(n^2)$, and we are done; so it remains to consider the cases where the first and second terms are largest. 
    
    \case{1}{The first term is largest and $t = 8mn/(c \cdot \sfE(\cP_1, \cP_2))$} Then we get \[\sfE(\cP_1, \cP_2) = O(t^{4/11}m^{12/11}n^{18/11 + \eta}) = O\left(\frac{m^{16/11}n^{2 + \eta}}{\sfE(\cP_1, \cP_2)^{4/11}}\right).\] Multiplying both sides by $\sfE(\cP_1, \cP_2)^{4/11}$ and raising them to the $(11/15)$th power gives \[\sfE(\cP_1, \cP_2) = O(m^{16/15}n^{22/15 + 11\eta/15}),\] which is at most  the first term in Lemma \ref{lem:final} (our choice of $\eta$ satisfies $11\eta/15 \leq \eps$). 
    
    \case{2}{The first term is largest and $t = 1/m$} Then we get 
    \begin{equation}
        \sfE(\cP_1, \cP_2) = O(t^{4/11}m^{12/11}n^{18/11 + \eta}) = O(m^{8/11}n^{18/11 + \eta}). \label{eqn:strange-bound}
    \end{equation} 
    We will show that the right-hand side must be bounded by one of the terms in Lemma \ref{lem:final}. First, if $m \leq n^{1/2 - 2\eta}$, then we have \[m^{8/11}n^{18/11 + \eta} \leq n^{4/11 - 16\eta/11 + 18/11 + \eta} \leq n^2.\] On the other hand, we claim that if $m \geq n^{1/2 - 2\eta}$, then it is bounded by the first term of Lemma \ref{lem:final}. To see this, we can write \[\frac{m^{8/11}n^{18/11 + \eta}}{m^{16/15}n^{22/15 + \eps}} = m^{-56/165}n^{28/165 + \eta - \eps} \leq n^{-28/165 + 112\eta/165 + 28/165 + \eta - \eps} \leq 1\] (since we chose $\eta = \eps/2$). So in either case, Lemma \ref{lem:final} holds. 

    \case{3}{The second term is largest and $t = 8mn/(c \cdot \sfE(\cP_1, \cP_2))$} Then we get \[\sfE(\cP_1, \cP_2) = O(t^{1/3}m^{4/3}n^{4/3}) = O\left(\frac{m^{5/3}n^{5/3}}{\sfE(\cP_1, \cP_2)^{1/3}}\right).\] Moving $\sfE(\cP_1, \cP_2)^{1/3}$ to the left-hand side and raising both sides to the $(3/4)$th power gives \[\sfE(\cP_1, \cP_2) = O(m^{5/4}n^{5/4}),\] which is the second term in Lemma \ref{lem:final}. 
    
    \case{4}{The second term is largest and $t = 1/m$} Then we get \[\sfE(\cP_1, \cP_2) = O(t^{1/3}m^{4/3}n^{4/3}) = O(mn^{4/3}).\] But in order to have $t = 1/m$, we must have $1/m \geq 8mn/(c \cdot \sfE(\cP_1, \cP_2))$, so this means \[m \leq \frac{c \cdot \sfE(\cP_1, \cP_2)}{8mn} = O(n^{1/3}).\] Then we have $mn^{4/3} = O(n^2)$, so Lemma \ref{lem:final} still holds in this case. 
    
    This means Lemma \ref{lem:final} is true in all possible cases, so we are done. 
\end{proof}

Finally, combining Lemma \ref{lem:final} with \eqref{eqn:cauchy--schwarz} gives that \[\abs{\Delta(\cP_1, \cP_2)} \geq \frac{m^2n^2}{\sfE(\cP_1, \cP_2)} = \Omega(\min\{m^{14/15}n^{8/15 - \eps}, m^{3/4}n^{3/4}, m^2, n^2\}),\] completing the proof of Lemma \ref{lem:cleaned-up-main}. 

\section*{Acknowledgements}

This project was conducted as part of the 2024 NYC Discrete Math REU, funded by NSF awards DMS-2051026 and DMS-2349366 and by Jane Street.

\bibliography{refs}{}
\bibliographystyle{plain}

\end{document}